\documentclass[12pt,reqno]{amsart}
\usepackage[headings]{fullpage}
\usepackage{amsmath,amssymb,amsthm,amsfonts,mathrsfs,mathtools,bbm}
\usepackage[all,cmtip]{xy}
\usepackage{url}
\usepackage[bookmarks=true,%
    colorlinks=true,%
    linkcolor=blue,%
    citecolor=blue,%
    filecolor=blue,%
    menucolor=blue,%
    urlcolor=blue,%
    breaklinks=true]{hyperref}
\usepackage{slashed}    
\usepackage{verbatim}
\usepackage{graphicx}
\usepackage{caption}
\usepackage{subcaption}

\usepackage[normalem]{ulem}

\newtheorem{theorem}{Theorem}[section]
\theoremstyle{definition}
\newtheorem{claim}[theorem]{Claim}

\newtheorem{lemma}[theorem]{Lemma}

\newtheorem{remark}[theorem]{Remark}

\newtheorem{conjecture}[theorem]{Conjecture}

\def\BZ{\mathbbm Z}
\def\BQ{\mathbbm Q}
\def\BR{\mathbbm R}
\def\BC{\mathbbm C}

\def\SL{\mathrm{SL}}

\def\PSL{\mathrm{PSL}}

\def\Im{\mathrm{Im}}

\def\be{\begin{equation}}
\def\ee{\end{equation}}

\def\SL{\mathrm{SL}}
\def\tr{\mathrm{tr}}
\def\NT{\mathcal{N}}

\usepackage{xparse}
\NewDocumentCommand{\ceil}{s O{} m}{%
  \IfBooleanTF{#1} 
    {\left\lceil#3\right\rceil} 
    {#2\lceil#3#2\rceil} 
}

\NewDocumentCommand{\floor}{s O{} m}{%
  \IfBooleanTF{#1} 
    {\left\lfloor#3\right\rfloor} 
    {#2\lfloor#3#2\rfloor} 
}

\begin{document}
\title[On the trace fields of hyperbolic Dehn fillings]{
       On the trace fields of hyperbolic Dehn fillings}
\author{Stavros Garoufalidis}
\address{
  International Center for Mathematics, Department of Mathematics \\
  Southern University of Science and Technology \\
  Shenzhen, China \newline
  {\tt \url{http://people.mpim-bonn.mpg.de/stavros}}}
\email{stavros@mpim-bonn.mpg.de}
\author{BoGwang Jeon}
\address{Department of Mathematics, POSTECH\\
77 Cheong-Am Ro, Pohang, South Korea}
\email{bogwang.jeon@postech.ac.kr}
\thanks{{\em Key words and phrases}: cusped hyperbolic 3-manifolds, Dehn-filling, trace field, A-polynomial, Newton polygon, Mahler's measure, Lehmer's Conjecture, fewnomials.}
\thanks{{\em 2020 Mathematics Subject Classification}: 57K31, 57K32, 11R06.}

\date{29 February 2024}
\begin{abstract}
  Assuming Lehmer's conjecture, we estimate the degree of the trace field
  $K(M_{p/q})$ of a hyperbolic Dehn filling $M_{p/q}$ of a 1-cusped
  hyperbolic 3-manifold $M$ by
\begin{equation*}
  \frac{1}{C}(\max\;\{|p|,|q|\})\leq \text{deg }K(M_{p/q})
  \leq C(\max\;\{|p|,|q|\})
\end{equation*}
where $C=C_M$ is a constant that depends on $M$. 
\end{abstract}

\maketitle

{\footnotesize
\tableofcontents
}


\section{Introduction}
\label{sec.intro}

\subsection{Main result}
An important arithmetic invariant of a complete hyperbolic 3-manifold $M$ of finite
volume
is its trace field $K(M)=\BQ(\tr \;\rho(g) \,\, | \,\, g \in \pi_1(M))$ generated by
the traces $\rho(g)$ of the elements of the fundamental group $\pi_1(M)$ of the
geometric representation $\rho: \pi_1(M) \to \PSL_2(\BC)$. Mostow rigidity implies
that $\rho$ is rigid, hence it can be conjugated to lie in $\PSL_2(\overline{\BQ})$,
and as a result, it follows that $K(M)$ is a number field. 

Given a 1-cusped hyperbolic 3-manifold $M$, Thurston showed that the manifolds
$M_{p/q}$ obtained by $p/q$-Dehn filling on $M$ are hyperbolic for all but finitely
many pairs of coprime integers $(p,q)$~\cite{Thurston}. Thus, it is natural to
ask how the trace field $K(M_{p/q})$ depends on the Dehn-filling parameters.
In \cite{hod}, C. Hodgson proved the following.

\begin{theorem}[Hodgson]
  Let $M$ be a $1$-cusped hyperbolic $3$-manifold. Then there are only
  finitely many hyperbolic Dehn fillings of $M$ of bounded trace field
  degree. 
\end{theorem}
Additional proofs of this theorem were given by Long--Reid~\cite[Thm.3.2]{LR}
and by the second author in his thesis~\cite{Jeon:thesis}.
A simple invariant of a number field is its degree. Hence, having the above theorem,
the next question is to ask for the behavior of the degree of the trace field of a
cusped hyperbolic $3$-manifold under Dehn filling. In this paper, we give a partial
and conditional answer to the question as follows:  

\begin{theorem}\label{main}
Let $M$ be a $1$-cusped hyperbolic $3$-manifold and $M_{p/q}$ be its $p/q$-Dehn filling. Assuming Lehmer's conjecture,
there exists $C=C_M$ depending on $M$ such that 
\begin{equation}\label{22033001}
\frac{1}{C}(\max\;\{|p|,|q|\})\leq \deg K(M_{p/q})\leq C(\max\;\{|p|,|q|\})
\end{equation}
\end{theorem}

Note that the upper bound in~\eqref{22033001} follows from a relatively easy
observation (see Theorems \ref{23083003}-\ref{21011903}) and so what really matters
in~\eqref{22033001} is its lower bound.

\subsection{Key Observation}\label{key}

The proof of the above theorem uses the fact that the geometric representation
of $M_{p/q}$ lies in the geometric component of the $\PSL_2(\BC)$-representation
variety of $M$ (known as Thurston's hyperbolic Dehn-surgery
theorem~\cite{Thurston}). The latter defines an algebraic curve in
$\BC^* \times \BC^*$ (the so-called $A$-polynomial curve~\cite{CCGLS}) in
meridian-longitude coordinates, and a suitable point in its intersection
with $m^p \ell^q=1$ determines the geometric representation of $M_{p/q}$.
Thus, the determination of the degree of $K(M_{p/q})$ is reduced to a bound
for the irreducible components of the Dehn-filling polynomial, a polynomial
whose coefficients are independent of $(p,q)$ for large $|p|+|q|$. 

To explain the key idea further in detail, as a toy model, let us assume the
$A$-polynomial of a $1$-cusped hyperbolic $3$-manifold $M$ is simply given as 
\begin{equation}\label{22010402}
m\Big(l-\frac{1}{l}\Big)+1+\frac{1}{m}\Big(\frac{1}{l}-l\Big)=0.
\end{equation}
Then finding the intersection between~\eqref{22010402} and $m^p\ell^q=1$ is
equivalent to solving
\begin{equation*}
t^{-q}(t^{p}-t^{-p})+1+t^{q}(t^{-p}-t^p)=0,  
\end{equation*}
which is normalized as 
\begin{equation}\label{22010401}
t^{2p+2q}-t^{2p}-t^{p+q}-t^{2q}+1=0
\end{equation}
under $0<q<p$. Since the Mahler measure of a polynomial concerns about the size of its roots (see Section \ref{19111503}), to apply it alongside Lehmer's conjecture (see Conjecture \ref{Lehmer}), we first bound the modulus of a root of \eqref{22010401} in terms of $p$ and $q$ as follows. Let $t_0$ is a root of~\eqref{22010401} with $|t_0|>1$. Then
\begin{equation*}
|t_0|^{2p+2q}<|t_0|^{2p}+|t_0|^{p+q}+|t_0|^{2q}+1< 4|t_0|^{2p}\Longrightarrow |t_0|^{2q}<4 
\end{equation*}
and, taking logarithms, 
\begin{equation*}
|t_0|<1+\frac{1}{q}.
\end{equation*}
If $S$ is an arbitrary constant with $\frac{p}{q}<S$, clearly
\begin{equation}\label{22010412}
|t_0|<1+\frac{S}{p}, 
\end{equation}
and this implies the Mahler measure of the minimal polynomial of $t_0$ is bounded above by  
\begin{equation*}
\Big(1+\frac{S}{p}\Big)^{\deg t_0}. 
\end{equation*}
According to Lehmer's conjecture, the above number is at least $1.176280818\dots$, thus there exists some constant $C$ depending only on $S$ such that\footnote{In this case, one may further obtain~\eqref{22010403}
  from~\eqref{22010412} unconditionally, thanks to Dimitrov's recent proof of the
  Schinzel-Zassenhaus conjecture. See Theorem \ref{dim} and Theorem \ref{22010601}.}  
\begin{equation}\label{22010403}
Cp<\deg\; t_0.
\end{equation}
Remark that~\eqref{22010403} already verifies Theorem \ref{main} for the given example partially over the following restricted domain of $p$ and $q$:
\begin{equation}\label{22033003}
\big\{(p,q)\in \mathbb{Z}^2\;:\;1<\frac{p}{q}<S\big\}, 
\end{equation}
and $C$, as a function of $S$, goes to $0$ as $S$ approaches $\infty$. 

Now suppose $\frac{p}{q}>S$. Let $t_0$ (with $|t_0|>1$) be a root of 
\begin{equation}\label{22010410}
t^{2p+2q}-t^{2p}-t^{p+q}-t^{2q}+1=t^{2p}(t^{2q}-1)-t^{p+q}-(t^{2q}-1)=0
\end{equation}
and $\delta_{p,q}$ be
\begin{equation}\label{22033005}
  \min \big\{|t_0^{2q}-1|\;|\;t_0\text{ is a root of }~\eqref{22010410}
  \text{ with }|t_0|>1\big\}.
\end{equation}
Then 
\begin{equation}\label{22040201}
\begin{gathered}
\delta_{p,q}|t_0^{2p}|(<|t_0^{2p}||t_0^{2q}-1|)<|t_0^{p+q}|+|t_0^{2q}|+1<3|t_0^{p+q}|\\
\Longrightarrow|t_0^{p-q}|<\frac{3}{\delta_{p,q}}\Longrightarrow
|t_0^{(1-\frac{1}{S})p}|<\frac{3}{\delta_{p,q}}\Longrightarrow |t_0|<1+\frac{D_{p,q}}{p}
\end{gathered}
\end{equation}
for some $D_{p,q}$ depending on $S$ and $\delta_{p,q}$. Hence, combining with
Lehmer's conjecture again, it follows that 
\begin{equation}\label{22040202}
C_{p,q}p<\deg\;t_0
\end{equation}
with $C_{p,q}$ depending on $D_{p,q}$. Note that, when $S$ is fixed to be sufficiently large, $D_{p,q}$ depends only on $\delta_{p,q}$, and 
\begin{equation*}
\lim_{\delta_{p,q}\rightarrow 0}D_{p,q}=\infty,\quad \lim_{D_{p,q}\rightarrow \infty}C_{p,q}=0. 
\end{equation*}
Consequently, if 
\begin{equation}\label{22040207}
\inf _{\substack{(p,q)\in \mathbb{Z}^2\\0<q<p}}\delta_{p,q}>0,
\end{equation}
the statement of Theorem \ref{main} holds over
\begin{equation}\label{22040203}
\big\{(p,q)\in \mathbb{Z}^2\;:\;0<q<p\big\} 
\end{equation}
for the example provided. 

However,~\eqref{22040207} is not generally true; in fact, the following holds:
\begin{equation*}
\inf _{\substack{(p,q)\in \mathbb{Z}^2\\0<q<p}}\delta_{p,q}=0.
\end{equation*}
Moreover, for any sufficiently large $D$, we always find $(p,q)$ in~\eqref{22040203}
and roots of~\eqref{22010410} whose absolute values are bigger than $1+\frac{D}{p}$,
which means the above heuristic argument is not enough to establish the claim of
Theorem \ref{main}. This aspect constitutes a highly non-trivial point of the problem.

To resolve the issue, in Lemmas \ref{21013003}-\ref{21011701}, we fix some
sufficiently large $D$ and count the number of roots of~\eqref{22010410}  whose
absolute values are larger than $1+\frac{D}{p}$. We further get the lower and upper
bounds of those roots and examine their distribution within the range. It is then
verified that the product of the root moduli stays in manageable limits, making it
small enough to deduce the desired outcome from Lehmer's conjecture.  

The key technical heart of the paper lies in proving Lemmas
\ref{21013003}-\ref{21011701}, and this requires a thorough examination of the local
geometry of an analytic curve. Employing solely elementary methods such as
trigonometry and basic analysis, we investigate the desired properties of the curve
in the proofs. We particularly encourage readers to contrast the content in this
subsection with the assertions given in Lemma \ref{21011701}.  

Finally let us remark that once the claim of Theorem \ref{main} is demonstrated over
the domain in~\eqref{22040203}, the rest will follow naturally by the symmetric
properties of the $A$-polynomial, along with a change of variables applied to it. 

\subsection{Acknowledgments}
We would like to thank the referee for their careful reading of the paper and for
various useful comments and suggestions.

\section{Preliminaries}

\subsection{The $A$-polynomial}
\label{A-poly}
For a given $1$-cusped hyperbolic $3$-manifold $M$, the $A$-polynomial of
$M$ is a polynomial with $2$-variables introduced by
Cooper--Culler--Gillett--Long--Shalen in \cite{CCGLS}. The subject has been studied
in great detail, as it provides much topological information of $M$. We will not
present all the technical details about the topic in this paper, but will instead
provide a brief outline of its construction and necessary properties that will be
utilized later in proving the main theorem. For a more comprehensive description of
the $A$-polynomial, we suggest readers refer to, for instance, \cite{Cha}
or \cite{CCGLS}.  

According to Thurston \cite{Thurston}, a geometric ideal triangulation $\mathcal{T}$
on $M$ induces a so-called \textit{gluing variety} $G(\mathcal{T})$ of $M$. The
variety represents the required conditions for how the tetrahedra in $\mathcal{T}$
are glued together along their edges to get a hyperbolic structure on $M$. Roughly,
$G(\mathcal{T})$ is seen as the set of all the possible hyperbolic structures on $M$,
or simply, the moduli space of $M$. 

If $T$ is a torus cross-section of the cusp of $M$ and $\mu, \lambda$ are the
chosen meridian-longitude pair of $T$,  then each point of $G(\mathcal{T})$ gives
rise to a (Euclidean) similarity structure on $T$, thus inducing the following
\textit{holonomy map} 
\begin{equation}
\pi_1(T)\longrightarrow \text{Aff}(\mathbb{C}):=\{az+b\;:\;a\neq 0, b\in \mathbb{C}\}.
\end{equation}
Consequently, the dilation components of the holonomies of $\mu, \lambda$ produce
rational functions $m, l$ respectively on $G(\mathcal{T})$. Now the
\textit{$A$-polynomial} of $M$ is defined as the Zariski closure of the image under
\begin{equation*}
G(\mathcal{T})\longrightarrow (m,l).
\end{equation*}
The $A$-polynomial turns out to be independent of the triangulation $\mathcal{T}$,
but depends only on the fundamental group of $M$ as well as the chosen
meridian-longitude pair of $T$. 

The following properties are well-known: 

\begin{theorem}
  \cite{CCGLS}\label{21011606}
  Let $A(m,\ell) \in \BZ[m,\ell]$ be the $A$-polynomial of a $1$-cusped hyperbolic
  $3$-manifold $M$. 
\begin{enumerate}
\item
  $A(m,\ell)=\pm A(m^{-1},\ell^{-1})$ up to powers of $m$ and $l$. 
\item
  $(m,\ell)=(1, 1)$ is a point on $A(m,\ell)=0$, which gives rise to a
  discrete faithful representation of $\pi_1(M)$. 
\end{enumerate}
\end{theorem}
Writing
\begin{equation}
  \label{21013101}
A(m,\ell)=\sum_{i,j}c_{i,j}m^i\ell^j,
\end{equation}
the Newton polygon $\NT(A)$ of $A(m,\ell)$ is defined as the convex hull
in the plane of the set $\{(i,j)\;:\;c_{i,j}\neq 0\}$.

\begin{theorem}
\cite{CCGLS}\cite{CL}\label{21012801}
Let $M, A(m,\ell)=0$ and $\NT(A)$ be the same as above. Suppose $A(m,\ell)$
is normalized so that the greatest common divisor of the coefficients is $1$. 
\begin{enumerate}
\item
  If $(i,j)$ is a corner of $\NT(A)$, then $c_{i,j}=\pm 1$. 
\item
  For each edge of $\NT(A)$, the corresponding edge-polynomial is a
  product of cyclotomic polynomials.
\end{enumerate}
\end{theorem}

\subsection{The Dehn filling polynomial}
\label{sub.dehn}
Fix a $1$-cusped hyperbolic $3$-manifold $M$ with a meridian and longitude
$\mu$ and $\lambda$. Let $M_{p/q}$ denote its $p/q$-Dehn filling where (here and throughout) $(p,q)$ represents a pair of coprime integers. Thurston's hyperbolic Dehn-surgery theorem~\cite{Thurston} implies that $M_{p/q}$ is a hyperbolic
manifold for all but finitely many pairs $(p,q)$, or equivalently for almost
all pairs $(p,q)$, or yet equivalently, for all pairs $(p,q)$ with $|p|+|q|$ large. 
It follows by the Seifert-Van Kampen theorem that
$\pi_1(M_{p/q})=\pi_1(M)/(\mu^p \lambda^q=1)$. Hence if the $A$-polynomial of
$M$ is given by $A(m,\ell)=0$, a discrete faithful representation
$\phi:\pi_1(M_{p/q})\longrightarrow \PSL_2(\mathbb{C})$ is obtained by finding a
point on 
\begin{equation}
  \label{21011801}
\big(A(m,\ell)=0\big)\cap \big(m^p\ell^q=1\big),
\end{equation}
which is reduced to an equation $A_{p,q}(t)=0$ for a one-variable 
Dehn-filling polynomial $A_{p,q}(t)=A(t^{-q}, t^p)$. More precisely, if~\eqref{21013101} is normalized as  
\begin{equation}
  \label{21020202}
A(m,\ell)=\sum\limits_{j=0}^{n} \Big(\sum\limits_{i=a_j}^{b_j} c_{i,j}m^i\Big)\ell^j
\end{equation}
where $a_j, b_j$ and $c_{i,j}$ are integers with $c_{a_j,j}, c_{b_j,j} \neq 0$
for all $j$, then the Dehn-filling polynomial is given by
\begin{equation}
  \label{21013103}
  A_{p,q}(t) =
  \sum\limits_{j=0}^{n} \Big(\sum\limits_{i=a_j}^{b_j} c_{i,j}t^{-qi}\Big)t^{pj}.
\end{equation}


We now discuss some elementary properties of the Dehn-filling polynomial.
Let 
\begin{equation}
  \label{21020203}
S_A:=\max_{0\leq j\leq n-1}\Big\{\frac{a_n-a_j}{n-j}\Big\} 
\end{equation}
denote the largest slope of the Newton polygon of the $A$-polynomial
(see Figure \ref{polygon2}). 
\begin{figure}
\centering
   \includegraphics[width=0.4\linewidth]{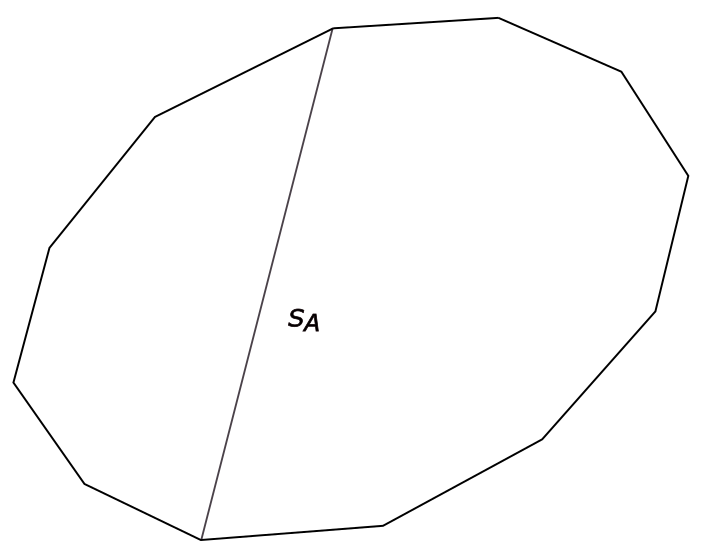}
  \caption{For instance, if the Newton polygon of the $A$-polynomial $M$ appears as above, then $S_A$ is the slope of the depicted edge.}\label{polygon2}
\end{figure}

\begin{lemma}\label{21011905}
When $|p|+|q|$ are large,
\begin{itemize}
\item[(a)]
  no two terms in~\eqref{21013103} are equal,
\item[(b)]
  further, if $q>0$, the leading term of $A_{p,q}(t)$ is of the form
  $c_{a_k,k} t^{-a_k q + k p}$ for some $0 \leq k \leq n$.
  Moreover, $k=n$ for $p/q >S_A$ and $q>0$.
\end{itemize}
\end{lemma}

It follows that for almost all pairs $(p,q)$, $A_{p,q}(t)$ has degree
piece-wise linear in $(p,q)$, and leading coefficient
(due to Theorem~\ref{21012801}) $\pm 1$. 

\begin{proof}
  Observe that $t^{-qi+pj}=t^{-qi'+pj'}$ implies that $p/q=(i'-i)/(j'-j)$ which takes
  finitely many values according to~\eqref{21020202}. It follows that when $|p|+|q|$
  is sufficiently large, then no two terms in~\eqref{21013103} are
  equal, and hence the Newton polygon of $A_{p,q}(t)$ is the image of the Newton
  polygon of $A(m,\ell)$ under the linear map $\BR^2 \to \BR$ which sends $(i,j)$ to
  $-qi+pj$. This proves that the leading term of $A_{p,q}(t)$ is the image under
  the above map of a corner of $A(m,\ell)$. When $p/q >S_A$ and $q>0$, it is easy to
  see that this corner of $A(m,\ell)$ is $(a_n,n)$.
\end{proof}


We now relate the trace field $K(M_{p/q})$ with~\eqref{21011801}.

\begin{theorem}\label{23083003}
  Let $(m,\ell)=(t_{0}^{-q}, t_{0}^{p})$ be the point on~\eqref{21011801} that gives
  rise to the discrete faithful representation of $\pi_1(M_{p/q})$.  (Of course,
  $t_0$ depends on $p$ and $q$.)  Then there exist constants $C=C(M)$ depending only on $M$ such that 
  \begin{equation}
    \label{21011807}
\frac{1}{C}\deg t_0\leq \deg K(M_{p/q})\leq C\deg t_0.
\end{equation}
\end{theorem}

\begin{proof}
  First note that $K(M)$ is generated by
  the traces of products of at most three generators of $\pi_1(M)$ (see for instance~\cite{MR}). By Theorem 3.1
  in \cite{dunfield}, the trace of any generator of $\pi_1(M)$ is represented as an
  algebraic function of the traces of $\mu$ and $\lambda$. As 
\begin{equation*}
  \tr\; \mu=\pm\Big(\sqrt{m}+\frac{1}{\sqrt{m}}\Big) \quad\text{and}
  \quad  \tr\;\lambda=\pm\Big(\sqrt{l}+\frac{1}{\sqrt{l}}\Big), 
\end{equation*}
it follows that the trace of any $\pi_1(M)$ is an algebraic function of $m$ and
$\ell$. Since $\pi_1(M)$ is finitely generated, the result follows.
\end{proof}

It is well-known (for example, see~\cite{Granville}) in
number theory that $A_{p,q}(t)$ has a bounded number of cyclotomic factors
for any $p$ and $q$ (although of unknown multiplicity). Therefore, thanks to
Theorem~\ref{23083003}, proving Theorem~\ref{main} is reduced to showing the degree
of any non-cyclotomic factor of $A_{p,q}(t)$ is bounded both above and below by some
multiples of $\max\{|p|,|q|\}$ under Lehmer's conjecture. That is, Theorem~\ref{main}
is derived from the following:

\begin{theorem}
  \label{21011903}
  Fix $M$ as in Theorem~\ref{main}. Assuming Lehmer's conjecture, there exist
  $C=C(M)$ depending only on $M$ such that, for any
  non-cyclotomic integer irreducible factor $g(t)$ of $A_{p,q}(t)$, 
  \begin{equation}
    \label{gbounds}
\frac{1}{C}\max\{|p|, |q|\}\leq \deg g(t)\leq C\max\{|p|, |q|\}.
\end{equation}
\end{theorem}

The upper bound follows trivially from~\eqref{21013103}, and can
even be replaced by a piece-wise linear function of $(p,q)$, as
was commented after Lemma~\ref{21011905}. The difficult part is the
lower bound, and this requires an analysis of the roots of $A_{p,q}(t)$
near $1$, along with careful estimates, as well as Lehmer's conjecture.

\begin{remark}
  In some special circumstances, such as Dehn-filling on one cusp of the Whitehead
  link complement $W$, it is possible to prove unconditionally that the degree of
$K(W_{1/q})$ is given by an explicit piece-wise linear function
of $q$ (see Hoste--Shanahan~\cite{HS:twist}). More generally, one
can prove that the degree of $K(W_{p/q})$ for fixed $p$ and large $q$
is given by an explicit piece-wise linear function of $q$~\cite{FGH}.   
\end{remark}

\subsection{Mahler measure and Lehmer's conjecture}
\label{19111503}

The Mahler measure $\mathscr{M}(f)$ and length $\mathscr{L}(f)$ of an integer polynomial 
\begin{equation*}
f(x)=a_nx^n+\cdots+a_1x+a_0=a_n(x-\alpha_1)\cdots(x-\alpha_n) \in \BZ[x]
\end{equation*}
are defined by 
\begin{equation*}
\begin{gathered}
\mathscr{M}(f)=|a_n|\prod^n_{i=1}\text{max}(|\alpha_i|, 1) \quad \text{and}\quad \mathscr{L}(f)=|a_0|+\cdots+|a_n|
\end{gathered}
\end{equation*}
respectively. Then the following properties are standard \cite{Mahler}:
\begin{enumerate}
\item
  $\mathscr{M}(f_1f_2)=\mathscr{M}(f_1)\mathscr{M}(f_2)$
\item
  $\mathscr{M}(f)\leq \mathscr{L}(f)$ 
\end{enumerate}
where $f_1, f_2\in\BZ[x]$. One of the renowned unsolved
problems in number theory is the following conjecture proposed by D. Lehmer in 1930s:
\begin{conjecture}[D. Lehmer]\label{Lehmer}
The Mahler measure of any non-cyclotomic irreducible integer polynomial $f$ satisfies 
\begin{equation*}
\mathscr{M}(f)\geq 1.176280818\dots
\end{equation*}
where $1.176280818\dots$ is the Mahler measure of 
\begin{equation*}
x^{10}+x^9-x^7-x^6-x^5-x^4-x^3+x+1
\end{equation*}
and the smallest known Salem number.
\end{conjecture}

\subsection{Schinzel-Zassenhaus conjecture}\label{dimitrov}
Recently there has been some progress towards Lehmer's conjecture. In \cite{Dim},
V. Dimitrov proved the following theorem:
\begin{theorem}[Schinzel-Zassenhaus conjecture]\label{dim}
  Let $f(t)$ be a monic integer irreducible polynomial of degree $n > 1$. If $f(t)$
  is not cyclotomic, then 
\begin{equation*}
  \max_{t_0\;:\;f(t_0)=0} |t_0|\geq 2^{\frac{1}{4n}} =1+\frac{\log 2}{4n}
  +O\Big(\frac{1}{n^2}\Big).
\end{equation*}
\end{theorem}

As mentioned in Section \ref{key}, once we restrict our attention to a small domain
such as~\eqref{22033003}, the conclusion of Theorem~\ref{main} follows
unconditionally (i.e. without relying on Lehmer's conjecture) over the domain solely
based on the above theorem (see Theorem \ref{22010601}). However we do not believe
Theorem~\ref{dim} in general provides an unconditional proof of Theorem~\ref{main}
over the entire domain. Please also refer to Remark~\ref{NZ}.


\section{Bounds for the roots of the Dehn-filling polynomial}
\label{sec.moduli}

In this section we study the roots of the Dehn-filling polynomial $A_{p,q}(t)$
for sufficiently large $|p|+|q|$. Recall the constant $S_A$ from~\eqref{21020203}.
The next lemma shows that when $p/q>S_A$ with $p$ and $q$ sufficiently
large, the roots of $A_{p,q}(t)$ are near the unit circle in the
complex plane.

\begin{lemma}
\label{21012208}
\rm{(a)} Fix a positive constant $S_1$ with $S_1>S_A$. There exists $D>0$
  that depends on $S_1$ such that for any coprime pair
  $(p,q)\in \mathbb{N}^2$ satisfying $\frac{p}{q}>S_1$, and for any root
  $t_0$ of $A_{p,q}(t)$, we have
\begin{equation*}
|t_0|< 1+\frac{D}{q}  \,.
\end{equation*}
\rm{(b)} If in addition $S_2 > S_1$, $S_1<\frac{p}{q}<S_2$, and $t_0$
is any root of $A_{p,q}(t)$, then
\begin{equation*}
|t_0|<1+\frac{D}{p} 
\end{equation*}
for some $D$ that depends on $S_1$ and $S_2$ only.
\end{lemma}

\begin{proof}
  For $S_1>S_A$, the leading term of $A_{p,q}(t)$ is
  $c_{a_n, n}t^{-a_nq+np}$ and $c_{a_n,n}=\pm 1$ by Lemma~\ref{21011905}.
    Thus
\begin{equation}
\label{20091102}
|t_0^{-a_nq+np}|< Lc_{\text{max}}|t_0^{-aq+bp}|\Longrightarrow |t_0^{(n-b)p+(a-a_n)q}|< Lc_{\text{max}}
\end{equation}
where $L$ is the number of terms of $A_{p,q}(t)$, $c_{\text{max}}$ is the maximum among all the coefficients of $A_{p,q}(t)$ and $-aq+bp$ is the second largest exponent of $A_{p,q}(t)$. Note that $L$ and $c_{\text{max}}$ are independent of $p$ and $q$. 

We now consider two cases: $b=n$ and $b<n$. If $b=n$, then $|t_0^q|<Lc_{\text{max}}$,
implying
\begin{equation}
  \label{21022502}
|t_0|<1+\frac{D}{q}
\end{equation}
for some constant $D$ depending on $L$ and $c_{\text{max}}$.

If $b < n$, since $p,q>0$, Lemma~\ref{21011905} implies that
$(a,b)=(a_j,j)$ for some $0\leq j\leq n-1$. Thus  
\begin{equation*}
\begin{gathered}
  Lc_{\text{max}}>\big|t_0^{(n-j)p+(a_i-a_n)q}\big|=
  \Big|t_0^{(n-j)\big(p-\frac{a_n-a_j}{n-j}q\big)}\Big|>
  \Big|t_0^{(n-j)\big(S_1-\frac{a_n-a_j}{n-j}\big)q}\Big|
\end{gathered}
\end{equation*}
by~\eqref{20091102}. By the assumption, as $S_1$ is strictly larger than
$\frac{a_n-a_j}{n-j}$ for every $j$, we get
\begin{equation*}
|t_0^q|< D_1\Longrightarrow |t_0|<1+\frac{D}{q}
\end{equation*}
for some constant $D_1$ and $D$ that depend only on $S_1$. Part (b) follows easily from part (a).
\end{proof}

The next lemma is a model for the roots of the Dehn-filling polynomial outside
  the unit circle. Indeed, as we will see later (in the proof of
  Lemma~\ref{21011701}), after a change of variables, we will
  bring the equation $A_{p,q}(t)=0$ into the form $z^q \phi(z)^p=1$ where $\phi$ is an analytic function at $z=0$ with $\phi(0)=1$. Hence, the lemma below
is the key technical tool used to bound the roots of the
Dehn-filling polynomial. In a simplified form, note that the equation $z^q(1+z)^p=1$
has $p+q$ solutions in the complex plane for $p, q >0$. On the other hand, only a
fraction of them are near zero whereas at the same time $1+z$ is outside the unit
circle. Moreover, we have a bound for the size of such solutions. 

\begin{lemma}
  \label{21013003}
  Let $w=\phi(z)$ be an analytic function defined near $(z,w)=(0,1)$ and $\epsilon$ be a sufficiently small number. Then there exists $\gamma(\epsilon)>0$ such that, for
  every coprime pair $(p,q)\in \mathbb{N}^2$ with $p/q>\frac{1}{\epsilon}$, the
  number of $(z,w)$ satisfying 
  \begin{equation}
    \label{21012209}
z^qw^p=1, \quad |w|>1, \quad |w-1|<\epsilon
\end{equation}
is at most $2\big(\ceil[\Big]{\frac{\gamma(\epsilon)p}{2\pi q}}+1\big)q$.
Moreover, $\gamma(\epsilon)\rightarrow 0$
as $\epsilon\rightarrow 0$ and, for each $h$ ($1\leq h\leq
\ceil[\Big]{\frac{\gamma(\epsilon)p}{2\pi q}}+1$), the product of the moduli of the first $2hq$ largest $w$ satisfying~\eqref{21012209} is bounded above by 
\begin{equation}
  \label{rootbound}
\prod\limits_{l=1}^{h} \bigg(1+\frac{d \log\frac{p/q}{l}}{p/q}\bigg)^{2q}
\end{equation}
where $d$ is some constant depending only on $\phi$. 
\end{lemma}
\begin{proof}
\begin{enumerate}
\item We first prove the lemma in the special case of $w=\phi(z)=1+z$. The first
  equation in~\eqref{21012209} is equivalent to 
\begin{equation}\label{210114003}
(1+z)^{p/q}=\frac{e^{2\pi ik/q}}{z}
\end{equation}
where $0\leq k\leq q-1$. 

\begin{figure}
\centering
\includegraphics[width=0.6\textwidth]{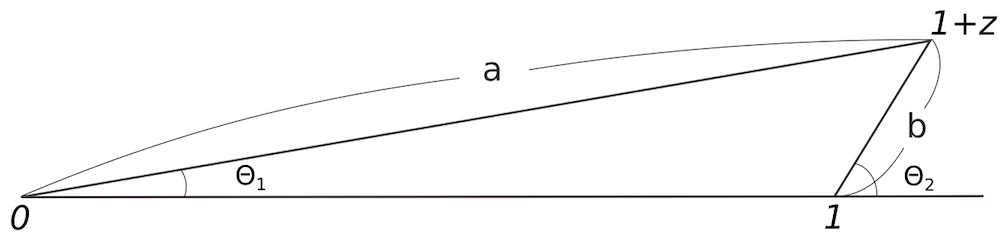}
\caption{$\bigtriangleup(z)$}\label{23091102}
\end{figure}
Considering $0, 1, 1+z$ as three vertices of a triangle $\bigtriangleup(z)$ in the
complex plane, we denote $1+z$ and $z$ by 
\begin{equation}\label{22040903}
\begin{gathered}
ae^{i\theta_1}:=1+z, \quad be^{i\theta_2}:=z 
\end{gathered}
\end{equation}
respectively where $a,b>0$ and $-\pi<\theta_1, \theta_2<\pi$
(Figure \ref{23091102}).\footnote{Note that if $\theta_1>0$ (resp. $\theta_1<0$),
  then $\theta_2>0$ (resp. $\theta_2<0$).}  Then~\eqref{210114003}
is equivalent to 
\begin{equation}\label{21012803}
(ae^{i\theta_1})^{p/q}=\frac{e^{2\pi ik/q}}{be^{i\theta_2}}\quad\Longrightarrow\quad  a^{p/q}=b^{-1},  \;\;\frac{\theta_1p}{q}\equiv \frac{2\pi k}{q}-\theta_2
  \quad (\noindent\mod 2\pi).
\end{equation}
By trigonometry, we have:
\begin{equation*}
\begin{gathered}
b^2=a^2+1-2a\cos \theta_1, \quad a^2=b^2+1+2b\cos \theta_2.
\end{gathered}
\end{equation*}
Since $b=a^{-p/q}$ by~\eqref{21012803}, it follows that 
\begin{equation}\label{210114006}
\begin{gathered}
  \cos \theta_1=\frac{1}{2a}(a^2+1-a^{-\frac{2p}{q}})
  =\frac{1}{2}\Big(a+\frac{1}{a}-\frac{1}{a^{\frac{2p}{q}+1}}\Big),\\
  \cos \theta_2=\frac{1}{2b}(a^2-b^2-1)
  =\frac{a^{\frac{p}{q}}}{2}(a^2-a^{\frac{-2p}{q}}-1)
  =\frac{1}{2}\Big(a^{\frac{p}{q}+2}-a^{\frac{-p}{q}}-a^{\frac{p}{q}}\Big).
\end{gathered}
\end{equation}
To simplify notation, we let\footnote{Since $|w|=a>1$ and $|z|=|ae^{i\theta_1}-1|<\epsilon$ by the assumption, we suppose $x>0$ and $\frac{x}{r}$ is sufficiently small.}  
\begin{equation}\label{22040905}
r:=p/q \quad \text{and}\quad  x:=r(a-1),
\end{equation}
and rewrite the equations
in~\eqref{210114006} as 
\begin{equation}\label{22040801}
\begin{gathered}
  \cos \theta_1=\frac{1}{2}\Bigg(1+\frac{x}{r}
  +\frac{1}{1+\frac{x}{r}}-\frac{1}{\big(1+\frac{x}{r}\big)^{2r+1}}\Bigg)
  =\frac{1}{2}\Bigg(2+\frac{x}{r}-\frac{\frac{x}{r}}{1+\frac{x}{r}}
  -\frac{1}{\big(1+\frac{x}{r}\big)^{2r+1}}\Bigg),\\
  \cos \theta_2=\frac{1}{2}\Bigg(\Big(1+\frac{x}{r}\Big)^{r+2}
  -\frac{1}{\big(1+\frac{x}{r}\big)^{r}}-\Big(1+\frac{x}{r}\Big)^{r}\Bigg),
\end{gathered}
\end{equation}
which implies 
\begin{equation*}
\begin{gathered}
  \sin^2\theta_1=1-\cos^2\theta_1=1-\Bigg(1+\frac{1}{2}\bigg(\frac{x}{r}
  -\frac{\frac{x}{r}}{1+\frac{x}{r}}-\frac{1}{\big(1
    +\frac{x}{r}\big)^{2r+1}}\bigg)\Bigg)^2\\
  =-\bigg(\frac{x}{r}-\frac{\frac{x}{r}}{1+\frac{x}{r}}
  -\frac{1}{\big(1+\frac{x}{r}\big)^{2r+1}}\bigg)
  -\frac{1}{4}\bigg(\frac{x}{r}-\frac{\frac{x}{r}}{1+\frac{x}{r}}
  -\frac{1}{\big(1+\frac{x}{r}\big)^{2r+1}}\bigg)^2\\
\Longrightarrow  
\sin \theta_1=\pm \sqrt{\bigg(\frac{1}{\big(1+\frac{x}{r}\big)^{2r+1}}
  +\frac{\frac{x}{r}}{1+\frac{x}{r}}-\frac{x}{r}\bigg)
    -\frac{1}{4}\bigg(\frac{1}{\big(1+\frac{x}{r}\big)^{2r+1}}
    +\frac{\frac{x}{r}}{1+\frac{x}{r}}
    -\frac{x}{r}\bigg)^2}.
\end{gathered}
\end{equation*}

Considering $\theta_1=\theta_1(x)$ as a function of $x$ and assuming
$\theta_1(x)\geq 0$ for the sake of simplicity, we briefly go over some basic behavior of the above function, set up the domain of
$x$ satisfying $|ae^{i\theta_1(x)}-1|=|z|<\epsilon$ and then list all the solutions
to the equation $z^q(1+z)^p=1$ over the domain found. 

Note that the following  
\begin{equation*}
\begin{gathered}
  \sqrt{\frac{1}{\big(1+\frac{x}{r}\big)^{2r+1}}
  +\frac{\frac{x}{r}}{1+\frac{x}{r}}
    -\frac{x}{r}}
  =\sqrt{\frac{1}{\big(1+\frac{x}{r}\big)^{2r+1}}
    -\frac{1}{\frac{r}{x}\big(\frac{r}{x}+1\big)}}
=\sqrt{\frac{1}{\big(1+\frac{x}{r}\big)^{2r+1}}-\frac{x^2}{r(x+r)}} 
\end{gathered}
\end{equation*}
is a decreasing function of $x$ and is $1$ when $x=0$. Further, if 
\begin{equation*}
\begin{gathered}
\frac{1}{(1+\frac{x}{r})^{2r+1}}=\frac{x^2}{r(x+r)},
\end{gathered}
\end{equation*}
then
\begin{equation}\label{21011304}
\begin{gathered}
\Big(\frac{r}{x+r}\Big)^{2r+1}=\frac{x^2}{r(x+r)}\Rightarrow
x^2=\frac{r^{2r+2}}{(x+r)^{2r}}
\Rightarrow x=\frac{r^{r+1}}{(x+r)^{r}}
\Rightarrow r=x\Big(1+\frac{x}{r}\Big)^r.
\end{gathered}
\end{equation}
And as $\Big(1+\frac{x}{r}\Big)^{\frac{r}{x}}$ is approximately equal to $e$
(simply say $\Big(1+\frac{x}{r}\Big)^{\frac{r}{x}}\approx e$) for $\frac{x}{r}$
sufficiently small, the root of~\eqref{21011304} is very close to the root of the
Lambert equation $r=xe^x$, studied in
great detail in~\cite{Corless}. Hence, for $x=\phi(r)$ satisfying~\eqref{21011304},
it follows that 
\begin{equation}\label{22040901}
\log r-\log\log r< \phi(r)< \log r-\log\log r+\log\log\log r.
\end{equation}
Remark that if $x=0$, then $\theta_1(0)=\frac{\pi}{3}, \theta_2(0)=\frac{2\pi}{3}$
and thus $\bigtriangleup(z=e^{\frac{2\pi}{3}i})$ is the unit equilateral triangle. As
$x$ increases, both $\theta_1(x)$ and $\theta_2(x)$ decrease, and, finally when
$x=\phi(r)$, $\theta_1(\phi(r))=\theta_2(\phi(r))=0$ and
$\bigtriangleup\Big(z=\frac{1}{\big(1+\frac{\phi(r)}{r}\big)^{r}}\Big)$ becomes
a flat triangle. 
\begin{enumerate}
\item By the assumption,  
\begin{equation*}
|z|=|ae^{i\theta_1(x)}-1|<\epsilon
\end{equation*}
and, as
\begin{equation*}
  |ae^{i\theta_1(x)}-1|=\Big|\Big(1+\frac{x}{r}\Big)e^{i\theta_1(x)}-1\Big|
  =\Big(1+\frac{x}{r}\Big)^2(2-2\cos \theta_1(x))
\end{equation*}
(with $\cos \theta_1(x)$ in~\eqref{22040801}) is a decreasing function of $x$,
there exists $\gamma=\gamma(\epsilon)$ depending on $\epsilon$ such that 
\begin{equation*}
  \Big|\Big(1+\frac{x}{r}\Big)e^{i\theta_1(x)}-1\Big|<\epsilon
  \Longleftrightarrow \gamma<x\leq \phi(r).
\end{equation*}
In conclusion, $(\gamma, \phi(r)]$ is a desired domain for $\theta_1(x)$ (and $\theta_2(x)$)
with the required property.
\item Now we list the solutions of $z^q(1+z)^p=1$ over the above domain. First the second equation
  in~\eqref{21012803} is reduced to 
\begin{equation}\label{21011501}
r\theta_1(x)+\theta_2(x)-\frac{2\pi k}{q}\in 2\pi\mathbb{Z}
\end{equation}  
and  
\begin{equation*}
  -\frac{2\pi k}{q} \leq r\theta_1(x)+\theta_2(x)
  -\frac{2\pi k}{q}\leq  r\theta_1(\gamma) +\theta_2(\gamma)-\frac{2\pi k}{q}
\end{equation*}
for $\gamma\leq  x\leq \phi(r)$. Since $r\theta_1(x)+\theta_2(x)$ is a decreasing
function of $x$ and $\theta_2(\gamma)\leq \frac{2\pi}{3}$, for each $k$ ($0\leq k\leq q-1$), the number
of $x$ satisfying~\eqref{21011501} is at most
\begin{equation*}
  \floor[\Big]{\frac{r \theta_1(\gamma)+\theta_2(\gamma)}{2\pi}-\frac{2\pi k}{q}}+1
  =\ceil[\Big]{\frac{r \theta_1(\gamma)}{2\pi} }+1.
\end{equation*}
Let $x^k_l\in (\gamma, \phi(r)]$ be a number satisfying 
\begin{equation*}
  \frac{r\theta_1(x^k_l)+\theta_2(x^k_l)-\frac{2\pi k}{q}}{2\pi}
  =l
\end{equation*}
where $0\leq l\leq \ceil[\Big]{\frac{r \theta_1(\gamma)}{2\pi} }$ and $l \in \mathbb{Z}$.
Then 
\begin{equation*}
  \theta_1(x^k_l)= \frac{2\pi l-\theta_2(x^k_l)+\frac{2\pi k}{q}}{r}
  \geq \frac{2\pi( l-\frac{1}{3})}{r} \geq \frac{l+1}{r}
\end{equation*}
for every $k$ ($0\leq k\leq q-1$) and $l\geq 1$. Using the fact that $\theta_1(x)$ is
a decreasing function, one further gets 
\begin{equation*}
x^k_l\leq \log \frac{r}{l+1}\quad \text{and so}\quad 1+\frac{x^k_l}{r}\leq 1+\frac{\log \frac{r}{l+1}}{r}\quad (l\geq 1).    
\end{equation*}
Clearly $x^k_0\leq \phi(r)\leq \log r$ (by~\eqref{22040901}) for any $0\leq k\leq q-1$. Consequently, for each $h$ ($1\leq h\leq
\ceil[\Big]{\frac{r \theta_1(\gamma)}{2\pi}}+1$),
the product of the moduli of the first $hq$ largest $w$ satisfying
$\theta_1(x), \theta_2(x)\geq 0$ and~\eqref{21011501} is bounded above
by\footnote{Recall from~\eqref{22040903} and~\eqref{22040905} that
  $|w|=|1+z|=|a|=1+\frac{x}{r}$.}
\begin{equation*}
  \prod\limits_{k=0}^{q-1}\prod\limits_{l=0}^{h-1}
  \bigg(1+\frac{x^k_l}{r}\bigg)\leq \prod\limits_{l=0}^{h-1}
  \bigg(1+\frac{\log\frac{p/q}{l+1}}{p/q}\bigg)^q=\prod\limits_{l=1}^{h}
  \bigg(1+\frac{\log\frac{p/q}{l}}{p/q}\bigg)^q.
\end{equation*}
Similarly, counting $x$ with $\theta_1(x), \theta_2(x)\leq 0$, one finally attains
the desired result. This concludes the
proof of the lemma for the special case of $w=1+z$. 
\end{enumerate}
\item To prove the lemma for the general case, suppose $w=\phi(z)$ is given as
  $\phi(z)=1+\sum\limits^{\infty}_{i=1} c_iz^i$ and let 
\begin{equation*}
  ae^{i\theta_1}:=1+\sum\limits^{\infty}_{i=1} c_iz^i,\quad
  be^{i\theta_2}:=\sum\limits^{\infty}_{i=1} c_iz^i, \quad ce^{i\theta_3}:=z.
\end{equation*}
We consider $b$ (resp. $\theta_2$) as an analytic function of $c$ (resp.
$\theta_3$). Thus, for $z$ sufficiently small, $b$ (resp. $\theta_1$) is
approximately very close to $|c_m|c^m$ (resp. $\arg c_m+m\theta_3$) where
$c_m$ is the coefficient of the first non-zero (and non-constant) term
of $\phi$.  Now~\eqref{21012209} is 
\begin{equation*}
(1+\phi(z))^{p/q}=\frac{e^{2\pi ik/q}}{z} \quad (0\leq k\leq q-1), 
\end{equation*}
which is equivalent to 
\begin{equation*}
c=a^{-r}, \quad r\theta_1+\theta_3\equiv\frac{2\pi ik}{q}\quad(\noindent \mod 2\pi i)
\end{equation*}
where $r=p/q$. Similar to the previous case, provided $x:=r(a-1)$,
we get $\frac{1}{c}=a^r=\big(1+\frac{x}{r}\big)^r$ and
\begin{equation*}
\begin{gathered}
  \sin \theta_1=\pm \sqrt{-\bigg(\frac{x}{r}
    -\frac{\frac{x}{r}}{1+\frac{x}{r}}-\frac{b^2}{a}\bigg)
    -\frac{1}{4}\bigg(\frac{x}{r}-\frac{\frac{x}{r}}{1+\frac{x}{r}}
    -\frac{b^{2}}{a}\bigg)^2}\\
  =\pm \sqrt{\bigg(\frac{b^{2}}{a}-\frac{x^2}{r(x+r)}\bigg)
    -\frac{1}{4}\bigg(\frac{b^{2}}{a}-\frac{x^2}{r(x+r)}\bigg)^2}.
\end{gathered}
\end{equation*}
Since $c=a^{-r}$ and $b\approx|c_m|c^m$ for $z$ small, $b \approx |c_m|a^{-rm}$ and so  
\begin{equation}
  \label{21013002}
\begin{gathered}
  \sin \theta_1\approx\pm \sqrt{\bigg(\frac{|c_m|^2}{a^{2rm+1}}
    -\frac{x^2}{r(x+r)}\bigg)-\frac{1}{4}\bigg(\frac{|c_m|^2}{a^{2rm+1}}
    -\frac{x^2}{r(x+r)}\bigg)^2}\\
  =\pm \sqrt{\bigg(\frac{|c_m|^2}{\big(1+\frac{x}{r}\big)^{2rm+1}}
    -\frac{x^2}{r(x+r)}\bigg)-\frac{1}{4}
    \bigg(\frac{|c_m|^2}{\big(1+\frac{x}{r}\big)^{2rm+1}}
    -\frac{x^2}{r(x+r)}\bigg)^2}
\end{gathered}
\end{equation}
for $z$ sufficiently small. Let $\phi(r)$ be the number satisfying
\begin{equation*}
\begin{gathered}
\frac{|c_m|}{\big(1+\frac{x}{r}\big)^{rm}}=\frac{x}{r}. 
\end{gathered}
\end{equation*}
As $\big(1+\frac{x}{r}\big)^{r}\approx e^x$ for $\frac{x}{r}$ small,
$\phi(r)$ is very close to the root of $|c_m|r=xe^{xm}$ and
\begin{equation*}
d_1\log r< \phi(r)< d_2\log r
\end{equation*}
for some $d_1, d_2\in \mathbb{Q}$. Further, one can check~\eqref{21013002}
is a decreasing function of $x$ over $0\leq x\leq \phi(r)$.
Applying similar methods used in the proof of the previous special case, we
obtain the desired result. 
\end{enumerate}
\end{proof}

In the next lemma, consider $M, M_{p/q}$ and the Dehn-filling polynomial
$A_{p,q}(t)$ as usual. As remarked earlier, the motivation for each statement
of Lemma \ref{21011701} was already explained in Section \ref{key} using a toy model.  

\begin{lemma} \label{21011701}
  For every $\epsilon>0$ sufficiently small, there exist $D(\epsilon)$
  and $\gamma(\epsilon)$ such that, for any coprime pair $(p,q)\in \mathbb{N}^2$ with $p/q>\frac{1}{\epsilon}$, the following hold.
\begin{enumerate}
\item
  If $t_0$ is a root of $A_{p, q}(t)$ such that $1<|t_0|^{p}<\frac{1}{\epsilon}$,
  then $|t_0|<1+\frac{D(\epsilon)}{p}$.
\item
  If $t_0$ is a root of $A_{p, q}(t)$ such that $|t_0|>1$ and
  $|t_0^{q}-\zeta|>\epsilon$ for every $\zeta$ satisfying
  $\sum\limits_{i=a_n}^{b_n}c_{i,n}\zeta^{-i}=0$, then
  $|t_0|<1+\frac{D(\epsilon)}{p}$.
\item
  There are at most $2\ceil[\Big]{\gamma(\epsilon)p/q}q$ roots of $A_{p,q}(t)=0$
  whose moduli are bigger than $1+\frac{D(\epsilon)}{p}$. Further, for
  each $h$ ($1\leq h\leq \ceil[\Big]{\frac{\gamma(\epsilon)p}{q}}+1$), the
  product of the moduli of the first $2hq$ largest roots of $A_{p, q}(t)$
  is bounded above by\footnote{Note that the exponent $2$ in~\eqref{21022601}
    is differed from the exponent $2q$ given in~\eqref{rootbound}. } 
  \begin{equation}
    \label{21022601}
\prod\limits_{l=1}^{h} \bigg(1+\frac{d \log\frac{p/q}{l}}{p/q}\bigg)^{2}
\end{equation}
where $d$ is some constant depending only on $M$. 
\item
  $D(\epsilon)\rightarrow \infty$ and
$\gamma(\epsilon)\rightarrow 0$ as $\epsilon \rightarrow 0$. 
\end{enumerate}
\end{lemma}

\begin{proof}
  For (1), if $|t_0|^{p}<\frac{1}{\epsilon}$ with $p$ sufficiently large, taking
  logarithms, there exists $D(\epsilon)$ such that $|t_0|<1+\frac{D(\epsilon)}{p}$. 

For (2), if $|t_0^{q}-\zeta|>\epsilon$ for every $\zeta$ satisfying
  $\sum\limits_{i=a_n}^{b_n}c_{i,n}\zeta^{-i}=0$, then there exists
  $\delta(\epsilon)>0$ such that
  $\Big|\sum\limits_{i=a_n}^{b_n}c_{i,n}t_0^{-iq}\Big|>\delta(\epsilon)$ and so
  \begin{equation}
    \label{21011603}
    \delta(\epsilon)|t_0^{np}|<\Big|\sum\limits_{i=a_n}^{b_n}c_{i,n}t_0^{-iq}
    \Big||t_0^{np}|<nL|t_0^{(n-1)p}|\Longrightarrow |t_0^{p}|<
    \frac{nL}{\delta(\epsilon)}
\end{equation}
where $L$ is some number bigger than
$\max\limits_{\substack{0\leq j\leq n-1}}
\Big|\sum\limits_{i=a_j}^{b_j}c_{i,j}t_0^{-iq}\Big|$. By~\eqref{21011603},
there exists some constant $D(\epsilon)$ depending on $\epsilon$ such that 
$|t_0|<1+\frac{D(\epsilon)}{p}$. 

For (3), suppose $t_0$ is a root of $A_{p,q}(t)=0$ such that 
  $|t_{0}^{q}-\zeta|<\epsilon$ where
  $\sum\limits_{i=a_n}^{b_n}c_{i,n}\zeta^{-i}=0$ and
  $|t_0|^{p}>\frac{1}{\epsilon}$ for some sufficiently small $\epsilon$. Note that
  $\zeta$ is a root of unity by Theorem~\ref{21012801} and, without loss of
  generality, we
  further assume $\zeta=1$. If $A(m,\ell)=0$ is the $A$-polynomial of
  $M$ as given in~\eqref{21020202}, one can view
  $(m, \ell)=(t_0^{-q}, t_0^{p})$ as an intersection point between
  $A(m,\ell)=0$ and $m^{p}\ell^{q}=1$. If we let $m':=\frac{1}{m}$
  and $f(m', \ell):=A(\frac{1}{m'}, \ell)$, then $(m',\ell)=(t_0^{q}, t_0^{p})$ is
  a point lying over $f(m', \ell)=0$ and $\ell^{q}=(m')^{p}$. For the sake
  of simplicity, we consider the projective closure of $f(m',\ell)=0$ and
  work with a different affine chart of it. More precisely, let $h(x',y',z')=0$ be the
  homogeneous polynomial representing the projective closure of
  $f(m', \ell)=0$ with $m'=\frac{x'}{z'}, \ell=\frac{y'}{z'}$. That is,
  $h(x',y',z')=0$ is obtained from
    $f\Big(\frac{x'}{z'}, \frac{y'}{z'}\Big)=0$ by multiplying by a power
    of $z'$ if necessary. Further, if $x:=\frac{x'}{y'},
  z:=\frac{z'}{y'}$ and $k(x,z):=h(x,1,z)$, since $x=\frac{m'}{\ell},
  z=\frac{1}{\ell}$ and
  \begin{equation*}
      \ell^{q}=(m')^{p}\Rightarrow \Big(\frac{y'}{z'}\Big)^{q}
    =\Big(\frac{x'}{z'}\Big)^{p}\Rightarrow z^{p-q}=x^{p},
\end{equation*}
$(x,z)=(t_0^{-p+q}, t_0^{-p})$ is an intersection point between $k(x,z)$ and
$z^{p-q}=x^{p}$. As $(t_0^{-p+q}, t_0^{-p})$ is sufficiently close to $(0,0)$, it
follows that $k(0,0)=0$ and thus $x$ is represented as an analytic
function $\varphi(z)$ of $z$ near $(0,0)$ (i.e. $k(\varphi(z), z)=0$)
with $(t_0^{-p+q}, t_0^{-p})\in (\varphi(z), z)$. Moreover, since
$\frac{t_0^{-p+q}}{t_0^{-p}}=t_0^{q}$ is close enough to $1$,
$\varphi(z)$ is of the following form:
\begin{equation*}
x=\varphi(z)=z\Big(1+\sum^{\infty}_{i=1}c_iz^i\Big).
\end{equation*}
Let $\phi(z):=1+\sum^{\infty}_{i=1}c_iz^i$. Then 
\begin{equation}
  \label{210114001}
\begin{gathered}
z^{p-q}=x^{p}
\Rightarrow z^{p-q}=z^{p}\phi(z)^{p} 
\Rightarrow z^{q}\phi(z)^{p}=1.
\end{gathered}
\end{equation}
In conclusion, $(t_0^{q}, t_0^{-p})$ is a point on
$(\frac{x}{z}, z)=(\phi(z), z)$ satisfying~\eqref{210114001}. Let $w:=\phi(z)$.
As $|t_0|>1$ and $|t_0^q-1|<\epsilon$ by the assumption, we get the desired
result by Lemma~\ref{21013003}. 

Finally, (4) is clear by the construction of $D(\epsilon)$ and
Lemma~\ref{21013003}. 
\end{proof}

\begin{remark}
  \label{21032001}
  Switching $p$ and $q$, one obtains an analogous result for a
  coprime pair $(p,q)$ with $q/p$ is sufficiently large. 
\end{remark}

\begin{remark}
  \label{NZ}
  Note that if $t_0$ (where $|t_0|>1$) is a root of $A_{p,q}(t)$ giving rise to
  the discrete faithful representation of $M_{p/q}$, for $|p|+|q|$ sufficiently
  large, $|t_0|$ is asymptotic to $1+\frac{2\pi \Im \;\tau}{|p+\tau q|^2}$ where
  $\tau$ is a complex number depending only on $M$ satisfying
  $\Im\;\tau\neq 0$ ($\tau$ is called the cusp shape of $M$)~\cite{NZ}.
  This, combining with the above remark, implies that $t_0$ is not the root of
  its minimal polynomial that appears in Theorem \ref{dim} for $|p|+|q|$
  sufficiently large.
\end{remark}


\section{Proof of the main theorem}
Now we prove the main theorem of the paper. Recall
that Theorem~\ref{main} follows from Theorem~\ref{21011903}. We first prove
a special case of Theorem~\ref{21011903} over a restricted domain of $p$ and
$q$. In particular, on this restricted domain, the statement of
Theorem \ref{21011903} holds unconditionally, without relying on Lehmer's
conjecture, thanks to Theorem \ref{dim}. We then expand the domain further and
prove Theorem \ref{21011903} over the extended range. We shall use
Lemma \ref{21011701} and assume Lehmer's conjecture from Theorem \ref{21012507}.
Once Theorem \ref{21012507} is proven, the rest of the proof of
Theorem \ref{21011903} will follow by symmetric properties of the $A$-polynomial
as well as a monomial change of variables on it. 

Finally let us remark again that, as the upper bound in Theorem \ref{21011903}
follows as a corollary of Bez\'{o}ut's theorem, only the lower bound
in~\eqref{gbounds} is proved. 

\begin{theorem}(without Lehmer)\label{22010601}
Let $M, M_{p/q}$ be as usual and $A_{p,q}(t),S_1, S_2$ be the
same as in Lemma~\ref{21012208}. Then there exists $C$ depending
on $S_1$ and $S_2$ such that, for any coprime $(p,q)\in \mathbb{N}^2$ satisfying
$S_1<\frac{p}{q}<S_2$ and any non-cyclotomic irreducible integer
factor $g(t)$ of $A_{p,q}(t)$, 
\begin{equation*}
C\max\{p,q\}\leq \deg g(t).
\end{equation*}
\end{theorem}
\begin{proof}
  By Lemma~\ref{21012208} (b), there exists $D$ depending on $S_1$ and $S_2$
  such that $|t_0|<1+\frac{D}{p}$ for any root $t_0$ of $A_{p,q}(t)=0$. By
  Theorem \ref{dim}, the result follows.  
\end{proof}

\begin{theorem}\label{21012507}(with Lehmer)
Let $M, M_{p/q}$ be as usual and $A_{p,q}(t),S_1$ be the
same as in Lemma~\ref{21012208}. Assuming Lehmer's conjecture, there exists $C$
depending on $S_1$ such that, for any coprime $(p,q)\in \mathbb{N}^2$ satisfying
$\frac{p}{q}>S_1$ and any non-cyclotomic irreducible integer
factor $g(t)$ of $A_{p,q}(t)$, 
\begin{equation}\label{22010603}
C\max\{p,q\}\leq \deg g(t).
\end{equation}
\end{theorem}
\begin{figure}
\centering
\includegraphics[width=0.4\textwidth]{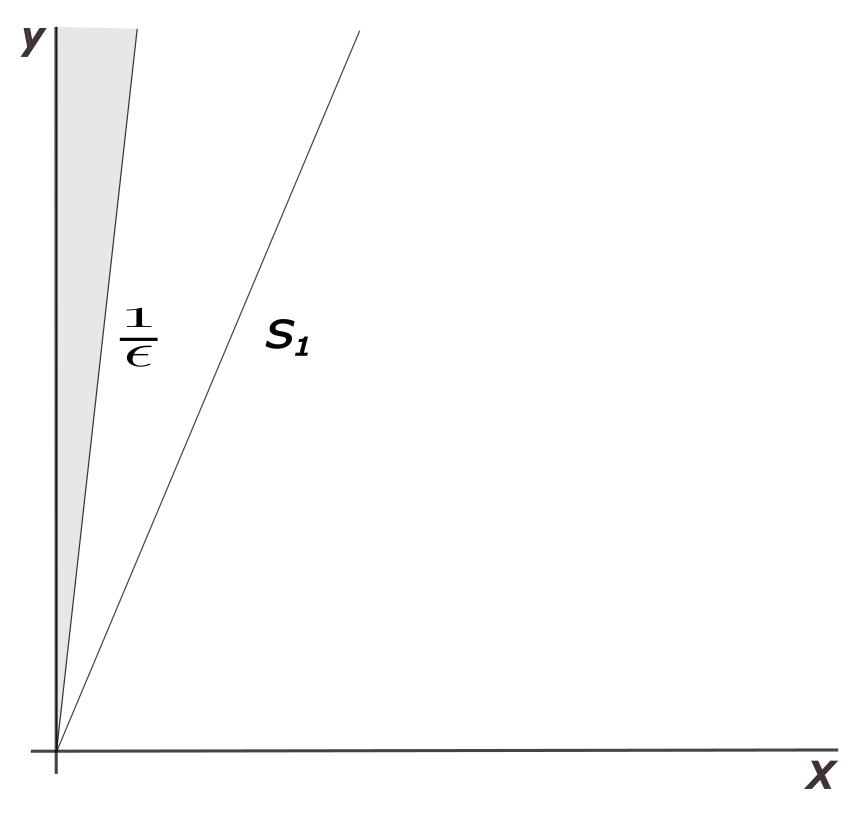}
\caption{In the proof of Theorem \ref{21012507}, to verify the claim for $(p,q)$
  contained in the shaded region above, we need the estimations carried out in
  Lemmas \ref{21013003}-\ref{21011701}.}
\end{figure}
\begin{proof}
  Let $\epsilon$ be some sufficiently small number such that $S_1<\frac{1}{\epsilon}$.
  By Theorem \ref{22010601}, we find $C'$ depending on $S_1$ and $\epsilon$
  satisfying 
\begin{equation*}
C'\max\{p,q\}\leq \deg g(t)
\end{equation*}
for any coprime $(p,q)\in \mathbb{N}^2$ with $S_1<\frac{p}{q}<\frac{1}{\epsilon}$
and any non-cyclotomic irreducible integer factor $g(t)$ of $A_{p,q}(t)$. 

Now suppose there are sequences of positive real numbers
  $\{c_i\}_{i\in \mathbb{N}}$ and coprime pairs $\{(p_i,q_i)\}_{i\in \mathbb{N}}$
  satisfying $\lim\limits_{i\rightarrow \infty}c_i=0$ and
  $\frac{p_i}{q_i}>\frac{1}{\epsilon}$.
  We further assume, for each $i$, there exists a non-cyclotomic irreducible
  integer factor $g_i(t)$ of $A_{p_i,q_i}(t)$ such that  
  \begin{equation}
    \label{21013110}
\deg g_i(t)\leq c_i\max\{p_i, q_i\}.
\end{equation}
We prove 
\begin{claim}\label{23083001}
\begin{equation}
  \label{21012501}
  \varlimsup_{i\rightarrow \infty}\mathscr{M}(g_i(t)) < 1.176280818\dots.
\end{equation}
\end{claim}
Once the above claim is established, the conclusion of Theorem~\ref{21012507}
can be readily deduced as follows. Since the Mahler measure of any non-cyclotomic
polynomial is strictly bigger than $1.176280818\dots$, by Lehmer's conjecture, this is a contradiction and thus there exists a constant
$C''$ depending on $\epsilon$ such that 
\begin{equation*}
C''\max\{p,q\}\leq \deg g(t) 
\end{equation*}
for any coprime $(p,q)\in \mathbb{N}^2$ satisfying $\frac{p}{q}>\frac{1}{\epsilon}$
and any non-cyclotomic irreducible integer factor $g(t)$ of $A_{p,q}(t)$. As
$\epsilon$ is arbitrarily, we may choose $C$ satisfying the statement of
Theorem~\ref{21012507}. 

\begin{proof}[Proof of Claim \ref{23083001}]
We now give the proof of~\eqref{21012501}. 
By Theorem ~\ref{21012801}, Lemma~\ref{21011905} and multiplying by a power of
$t$ if necessary, we assume both $A_{p_i,q_i}(t)$ and its irreducible factor
$g_i(t)$ are monic
  integer polynomials.
  
  As $c_i\rightarrow 0$, for $i$ sufficiently large, the product of the moduli of
  the first $\lfloor{c_ip_i\rfloor}$ largest roots of $A(t^{-q_i}, t^{p_i})=0$
  (and so $\mathscr{M}(g_i(t))$ as well) is bounded above by
  \begin{equation}
    \label{21011503}
  \prod_{l=1}^{\lceil{\frac{c_ip_i}{q_i}\rceil}}
  \Bigg(1+\frac{d\log \frac{p_i/q_i}{l}}{p_i/q_i}\Bigg)^2
\end{equation}
where $d$ is some constant depending only on $M$ by Lemma~\ref{21011701}. We
show~\eqref{21011503} is strictly less than $1.176280818\dots$ as $i\rightarrow \infty$. To
simplify notation, let $r_i:=p_i/q_i$.
Taking logarithms to $\mathscr{M}(g_i(t))$ and~\eqref{21011503},  
\begin{equation*}
  \log \mathscr{M}(g_i(t))\leq 2\sum_{l=1}^{\lceil{c_ir_i\rceil}}
  \log\Big(1+\frac{d\log \frac{r_i}{l}}{r_i}\Big).
\end{equation*}
Since
\begin{equation*}
\begin{gathered}
  2\sum_{l=1}^{\lceil{c_ir_i\rceil}}\log\Big(1
  +\frac{d\log \frac{r_i}{l}}{r_i}\Big)\leq 2\sum_{l=1}^{\lceil{c_ir_i\rceil}}
  \frac{d\log \frac{r_i}{l}}{r_i}
\end{gathered}
\end{equation*}
and
\begin{equation*}
\begin{gathered}
  \lceil{c_ir_i\rceil}\log \lceil{c_ir_i\rceil}-\lceil{c_ir_i\rceil}<\log
  \lceil{c_ir_i\rceil}!, 
\end{gathered}
\end{equation*}
it follows that 
\begin{gather}
  \log \mathscr{M}(g_i(t))\leq 2\sum_{l=1}^{\lceil{c_ir_i\rceil}}
  \frac{d\log \frac{r_i}{l}}{r_i}=\frac{2d\log
    \frac{r_i^{\lceil{c_ir_i\rceil}}}{\lceil{c_ir_i\rceil}!}}{r_i}=
  \frac{2d\lceil{c_ir_i\rceil}\log r_i-2d\log \lceil{c_ir_i\rceil}!}{r_i}
  \nonumber\\
  <\frac{2d\lceil{c_ir_i\rceil}\log r_i-2d\big(\lceil{c_ir_i\rceil}
    \log\lceil{c_ir_i\rceil}-\lceil{c_ir_i\rceil}\big)}{r_i}
  <\frac{2d\lceil{c_ir_i\rceil}\log
    \frac{r_i}{\lceil{c_ir_i\rceil}}+2d\lceil{c_ir_i\rceil}}{r_i}.
  \label{21013111}
\end{gather}
We now consider two cases.

\noindent{\bf Case 1.}
If $c_ir_i\geq 1$, then~\eqref{21013111} is bounded above by 
\begin{equation}
  \label{21013112}
\frac{4dc_ir_i\log\frac{1}{c_i}+4dc_ir_i}{r_i}=-4dc_i\log c_i+4dc_i.
\end{equation}
Since $\lim\limits_{i\rightarrow \infty}c_i=0$,~\eqref{21013112}
\big(resp.~\eqref{21011503}\big) converges to $0$ (resp. $1$) as
$i\rightarrow \infty$.  

\noindent{\bf Case 2.}
If $c_ir_i<1$, then~\eqref{21013111} is bounded above by
  \begin{equation}
    \label{21020101}
\frac{4d\log r_i+4d}{r_i}.
\end{equation}
As $r_i(=p_i/q_i)>\frac{1}{\epsilon}$ and $d$ depends only on $M$,~\eqref{21020101}
is strictly less than $\log 1.176280818\dots$ provided $\epsilon$ is sufficiently small. This completes the proof of~\eqref{21012501} as well as the proof of Theorem~\ref{21012507}.   
\end{proof}
\end{proof}

\begin{remark}
  \label{21020201}
For $p>0$ and $q<0$, one analogously gets the following result. 
  Suppose $A_{p,q}(t)$ is given as in~\eqref{21013103}. Let $S_1$ be some positive
  constant such that $S_1>\frac{b_j-b_n}{n-j}$ for every $j$ ($0\leq j\leq n-1$).
  Assuming Lehmer's conjecture, there exists $C$ depending on $S_1$ such that,
  for any coprime
  pair $(p,q)$ with $p>0, q<0, \frac{p}{|q|}>S_1$ and any non-cyclic irreducible
  factor $g(t)$ of $A_{p,q}(t)$, 
\begin{equation*}
C\max\{p,|q|\}\leq \deg g(t).
\end{equation*}
\end{remark}

We now use the $\SL_2(\BZ)$ action on the lattice $\BZ^2$ which amounts to
a monomial change of variables on $A(m,\ell)$ and hence on $(p,q)$ and on
$A_{p,q}(t)$. Combining this action with Theorem~\ref{21012507} and Remark~\ref{21020201}, we obtain the following.

\begin{theorem}(with Lehmer)
  \label{2101280}
  Let $M, M_{p/q}$ and $A_{p,q}(t)$ be as in Theorem~\ref{21012507}. Assuming Lehmer's
  conjecture, the following statements hold. 
\begin{enumerate}
\item
  If $\frac{a_n-a_j}{n-j}<0$ for all $0\leq j\leq n-1$, then there exists
  $C$ depending on $M$ such that, for any coprime pair
  $(p,q)\in \mathbb{N}^2$ and any non-cyclic irreducible factor $g(t)$ of $A_{p,q}(t)$, 
\begin{equation*}
C\max\{p,q\}\leq \deg g(t).
\end{equation*}
\item
  If $\frac{a_n-a_j}{n-j}>0$ for some $0\leq j\leq n-1$, then there exists $C$
  depending on $M$ such that, for any coprime pair
  $(p,q)\in \mathbb{N}^2$ with $\frac{p}{q}>S_A$ where $S_A$ is the one given
  in~\eqref{21020203} and any non-cyclic irreducible factor $g(t)$ of $A_{p,q}(t)$, 
\begin{equation*}
C\max\{p,q\}\leq \deg g(t).
\end{equation*}
\end{enumerate}
\end{theorem}
\begin{figure}

\centering
\includegraphics[width=0.4\textwidth]{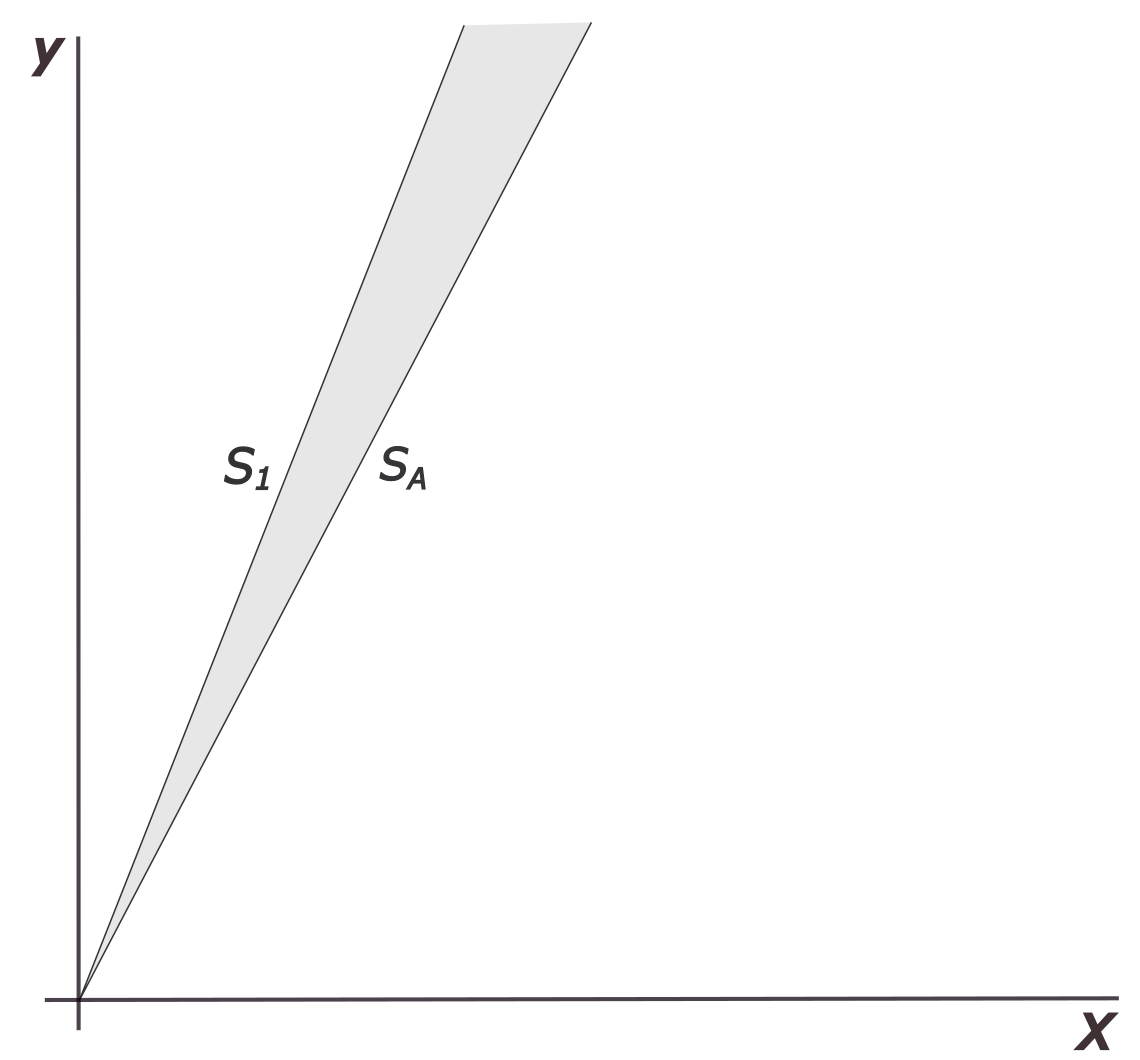}
\caption{We have established the claim for $(p, q)$ in the region bounded by the
  $y$-axis and the line with slope $S_1$, as proven in Theorem \ref{21012507}. To
  extend the claim further to the shaded region above, we apply a change of
  variables and reduce it to the previous case. Please see the proof of
  Theorem \ref{2101280}.}\label{23091101}
\end{figure}

\begin{proof}
  For (1), 
  if $\frac{a_n-a_j}{n-j}<0$ for all $0\leq j\leq n-1$, equivalently, it
  means $c_{a_n, n}t^{-a_nq+np}$ is the leading term of $A_{p,q}(t)$ for every
  coprime pair $(p,q)\in \mathbb{N}^2$ with $p+q$ sufficiently large. Provided
  $S_1$ is chosen to be some sufficiently small $\epsilon$, the claim is true
  for any coprime pair $(p,q)\in \mathbb{N}^2$ with $\frac{p}{q}>\epsilon$ by
  Theorem~\ref{21012507}. If $\frac{p}{q}<\epsilon $ (or, equivalently,
  $\frac{q}{p}>\frac{1}{\epsilon}$), by switching $p$ and $q$, one gets the
  desired result by an analogue of Lemma~\ref{21011701} (see Remark~\ref{21032001})
  and similar arguments given in the proof of
  Theorem~\ref{21012507}. 

  For (2), 
  suppose $\frac{a_n-a_j}{n-j}>0$ for some $j$ ($0\leq j\leq n-1$) and let $S_A$ be
  as in~\eqref{21020203}. Let $S_1:=S_A+\epsilon$ where $\epsilon$ some
  sufficiently small number. If $(p,q)\in \mathbb{N}^2$ with $\frac{p}{q}>S_1$,
  the result follows by Theorem~\ref{21012507}. For $(p,q)\in \mathbb{N}^2$
  satisfying $S_A<\frac{p}{q}<S_1$ (see Figure \ref{23091101}), first let $(a,b)$
  \big(resp. $(r,s)$\big) be
  a coprime pair such that $\frac{a}{b}=S_A$ (resp. $bs+ar=1$). We further assume
  $a,b>0$. Since $S_A(=\frac{a}{b})<\frac{p}{q}<S_1(=S_A+\epsilon)$,  
\begin{equation*}
  \Big|\frac{p}{q}-\frac{a}{b}\Big|<\epsilon \Longrightarrow
  \Big|\frac{bq}{aq-bp}\Big|>\frac{1}{\epsilon} 
\end{equation*}
and so 
\begin{equation}
  \label{21012601}
  \frac{rp+sq}{aq-bp}=\frac{-\frac{r}{b}(aq-bp)+\frac{bs+ar}{b}q}{aq-bp}=
  -\frac{r}{b}+\frac{1}{b}\frac{q}{aq-bp}, 
\end{equation}
implying 
\begin{equation}
  \label{21012704}
\Big|\frac{rp+sq}{aq-bp}\Big|>\frac{1}{b^2\epsilon}-\frac{r}{b}.
\end{equation}
As given in Section~\ref{A-poly}, let $\mu,\lambda$ be a chosen meridian-longitude
pair of $T$, a torus cross section of the cusp of $M$. By setting
$\mu':=\mu^{a}\lambda^{b}, \lambda':=\mu^{-s}\lambda^{r}$, we change basis of
$T$ from $\mu,\lambda$ to $\mu',\lambda'$.
Also let $A'(m', \ell')=0$ be the $A$-polynomial of $M$ obtained from the
new basis. Since $\mu=(\mu')^{r}(\lambda')^{-b}, \lambda=(\mu')^{s}(\lambda')^{a}$ and
$\mu^p\lambda^q=(\mu')^{rp+sq}(\lambda')^{-bp+aq}$, $\frac{p}{q}$-Dehn filling of $M$
under the original basis corresponds to $\big(\frac{rp+sq}{-bp+aq}\big)$-Dehn
filling of $M$ under the new basis. Let $p':=rp+sq, q':=-bp+aq$. Note that 
\begin{equation}
  \label{21013008}
p'>0,\quad  q'<0, \quad |p'/q'|>\frac{1}{b^2\epsilon}-\frac{r}{b} 
\end{equation}
by~\eqref{21012601}-\eqref{21012704}. Since $\frac{1}{b^2\epsilon}-\frac{r}{b}$
is sufficiently big, by Remark~\ref{21020201}, there is $C'$ such that,
for any $(p',q')$ satisfying~\eqref{21013008} and any non-cyclotomic irreducible
factor of $A'(t^{-q'},t^{p'})=0$, 
\begin{equation*}
C' \, p' \leq \deg g(t).
\end{equation*}
Equivalently, it means the degree of any non-cyclotomic irreducible factor of
$A_{p,q}(t)=0$ is bounded above and below by constant multiples of $rp+sq$. This completes the proof (recall that $r$ and $s$ are independent of $p$ and $q$). 
\end{proof}
 
Now we are ready to conclude the proof of Theorem~\ref{21011903}.

\begin{proof}[Proof of Theorem~\ref{21011903}]
We show the theorem only for $p,q>0$. The rest of the cases can be treated similarly. 
Let $A_{p,q}(t)$ be normalized as~\eqref{21013103}. For $\frac{a_n-a_j}{n-j}<0$
for every $j$ ($0\leq j\leq n-1$), the result follows by Theorem~\ref{2101280} (1)
and so it is assumed $\frac{a_n-a_j}{n-j}>0$ for some $j$ ($0\leq j\leq n-1$). Let
$S_A$ be given as in~\eqref{21020203}.
For a coprime pair $(p,q)$ satisfying $p/q>S_A$,
since the claim was also proved in Theorem~\ref{2101280} (2), we further suppose
$p/q<S_A$ and let 
\begin{equation*}
n_2:=\min\limits_{0\leq j\leq n-1}\Big\{j\;|\;S_A=\frac{a_n-a_j}{n-j}\Big\}.
\end{equation*}
Note that $(a_{n_2}, n_2)$ is a corner of $\NT(A)$ and $S_A$ is the slope
of the edge $E_{S_A}$ connecting two corners $(a_n, n)$ and $(a_{n_2}, n_2)$ of
$\NT(A)$.

\begin{figure}
\centering
\begin{subfigure}{.7\textwidth}
  \centering
  \includegraphics[width=0.7\linewidth]{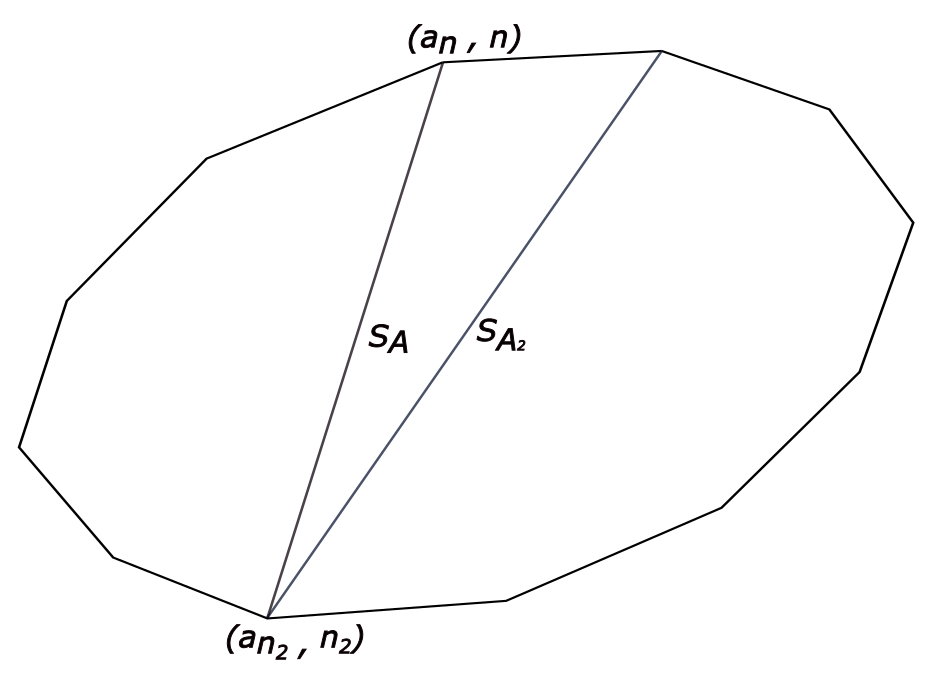}
  \subcaption{If the Newton polygon of the $A$-polynomial $M$ is the same as the one given in Figure \ref{polygon2}, then the corners $(a_n, n),(a_{n_2}, n_2)$ are shown as above in the picture, and $S_A, S_{A_2}$ are the slopes of the depicted edges.}
\end{subfigure}
\begin{subfigure}{.7\textwidth}
  \centering
  \includegraphics[width=0.6\linewidth]{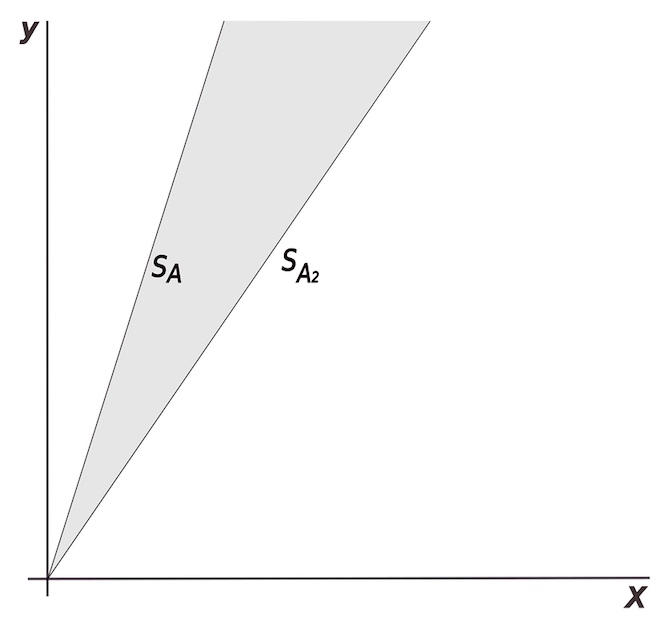}
  \subcaption{Once the claim is proven for $(p,q)$ over the sector between the
    $y$-axis and the edge of slope $S_A$ (as established in Theorem \ref{2101280}),
    the proof of it over the the shaded region is obtained through a change of
    variables.}
  \label{fig:sub2}
\end{subfigure}
\caption{}
\label{polygon}
\end{figure}
We distinguish two cases.

\noindent{\bf Case 1.}
If $n_2=0$ or $\frac{a_{n_2}-a_j}{n_2-j}<0$ for every $0\leq j\leq n_2-1$,
  then $c_{a_{n_2}, n_2}t^{-a_{n_2}q+n_2p}$ is the leading term of $f_{p,q}(t)$ for
  any $(p,q)\in \mathbb{N}^2$ with $\frac{p}{q}<S_A$ and $p+q$ sufficiently large.
  By interchanging $p$ and $q$, one gets the desired result following similar
  steps shown in the proof of Theorem~\ref{2101280} (2).

\noindent{\bf Case 2.}  
  Now suppose $n_2\neq 0$ and $\frac{a_{n_2}-a_j}{n_2-j}>0$ for some
  $0\leq j\leq n_2-1$. Let  
\begin{equation*}
S_{A_2}:=\max\limits_{0\leq j\leq n_2-1}\Big\{\frac{a_{n_2}-a_j}{n_2-j}\Big\}.
\end{equation*} 
Then $S_A>S_{A_2}$ and $S_{A_2}$ is the slope of the edge $E_{S_{A_2}}$ of $\NT(f)$
adjacent to $E_{S_A}$. To show the claim for $(p,q)\in \mathbb{N}^2$ satisfying
$S_{A_2}< \frac{p}{q}< S_A$, let
$T, \mu, \lambda, \mu'(=\mu^{a}\lambda^{b}), \lambda'(=\mu^{-s}\lambda^{r})$ be the
same as in the proof of Theorem~\ref{2101280} (2). Also we denote the
$A$-polynomial of $M$ obtained from $\mu', \lambda'$ by $A'(m', \ell')=0$ and
assume its $p'/q'$-Dehn filling equation $A'_{p',q'}(t):=A'(t^{-q'}, t^{p'})$
is given as 
\begin{equation*}
\begin{gathered}
\sum\limits_{j=0}^{n'} \Big(\sum\limits_{i=a'_j}^{b'_j} c'_{i,j}t^{-q'i}\Big)t^{p'j}
\end{gathered}
\end{equation*}
where $a'_j, b'_j\in \mathbb{Z}$. Under $p'=rp+sq$ and $q'=-bp+aq$, the set
of coprime pairs $(p,q)$ satisfying $S_{A_2}<\frac{p}{q}<S_A$ is transformed to
the set of $(p',q')$ satisfying $\frac{p'}{q'}>S_{A'}$ where 
\begin{equation*}
S_{A'}:=\max\limits_{0\leq j\leq n'-1}\frac{a'_{n'}-a'_j}{n'-j}.
\end{equation*}
Since the conclusion holds for any $A'_{p', q'}(t)$ with
$(p',q')\in \mathbb{N}^2, \frac{p'}{q'}>S_{A'}$ by Theorem~\ref{2101280},
equivalently, it holds for $A_{p, q}(t)$ with
$(p,q)\in \mathbb{N}^2, S_{A_2}<\frac{p}{q}<S_A$ as well. 

Analogously, using the same ideas, one can also prove the statement for
$(p,q)\in \mathbb{N}^2$ satisfying $S_{A_3}<\frac{p}{q}<S_{A_2}$ where $S_{A_3}$ is
the slope of the edge $E_{S_{A_3}}(\neq E_A)$ of $\NT(A)$ adjacent
to $E_{S_{A_2}}$. Since $\NT(A)$ has only finitely many edges, the desired result will follow eventually. 
\end{proof}

\bibliographystyle{hamsalpha}
\bibliography{biblio}

\end{document}